\documentclass{amsart}

\usepackage{graphicx} 
\usepackage{todonotes}
\usepackage[english]{babel}
\usepackage{hyperref,amssymb,amsmath,amsthm,color,mathtools,tikz,subfig,dirtytalk,csquotes,soul,orcidlink,esint,cleveref}
\usetikzlibrary{arrows}

\numberwithin{equation}{section}

\newtheorem{example}{Example}[section]
\newtheorem{ex}[example]{Example}
\newtheorem{dhef}[example]{Definition}
\newtheorem{lemma}[example]{Lemma}
\newtheorem{remark}[example]{Remark}
\newtheorem{prop}[example]{Proposition}
\newtheorem{theo}[example]{Theorem}
\newtheorem{cor}[example]{Corollary}

\def\elle#1{L^{#1}}
\def\elleom#1{L^{#1}(\Omega)}

\def\SobomX#1{W^{1,#1}_0(\Omega,X)}
\def\Mar#1{\mathbf{M}^{#1}(\Omega)}

\def\io{\int_\Omega}

\def\norma#1#2{ \|#1 \|_{#2}}

\def\N{\mathbb{N}}

\def\R{\mathbb{R}}

\newcommand{\abs}[1]{\left|#1\right|}

\DeclareRobustCommand{\rchi}{{\mathpalette\irchi\relax}}
\newcommand{\irchi}[2]{\raisebox{\depth}{$#1\chi$}} 

\title[Linear $X$-elliptic equations with measurable coefficients]{Existence, uniqueness, regularity and stability of solutions to linear $X$-elliptic equations with measurable coefficients}
\author{Marco Picerni
\orcidlink{0009-0004-4364-4831}}\email{mpicerni@sissa.it}
\address{SISSA, via Bonomea 265, 34136, Trieste, Italy}

\begin{document}

\begin{abstract}
We prove an existence and uniqueness result for solutions to linear $X$-elliptic equations with $\elle1$ data and zero Dirichlet boundary conditions. Such solutions depend continuously on the datum. Moreover, we show that an improvement in the summability of the data yields a corresponding improvement in the summability of the solutions, in a manner analogous to the one that occurs in the case of uniformly elliptic equations.
\end{abstract}

\maketitle

\section{Introduction}
$X$-elliptic operators are a specific class of second-order differential operators in divergence form associated to a suitable family of Lipschitz-continuous vector fields $X=\{X_1,\ldots,X_m\}$. This class of operators, which are typically degenerate elliptic, was first introduced in \cite{KL-Xelliptic_intro} as a generalization of uniformly elliptic operators in divergence form. since then, multiple results for $X$-elliptic operators have been developed (see, for instance, \cite{FSSC-ApproxImbeddingWeighted}, \cite{KL-Liouville}, \cite{Mazzoni-Green}, and \cite{Pinamonti_multiplicity}).

While much of the research has focused on operators with continuous coefficients (especially on the study of the Sublaplacian on Carnot Groups, see \cite{SubLapl_Hormander}, \cite{Bramanti_book}), this paper shifts attention to $X$-elliptic operators with measurable coefficients. 
In particular, we are interested in the properties of solutions to partial differential equations with data belonging to $\elle p$ for some $p\geq1$. 

In the uniformly elliptic case, results on existence, uniqueness and regularity of solutions to such PDEs are well established and closely related to the Poincar\'e and the Sobolev inequalities. We extend these classical results to the $X$-elliptic setting using the results on Poincar\'e and Sobolev inequalities in metric spaces \cite{HP-SobPoinc,Sobmetric-book} alongside classical technique \cite{BoccardoCroce,BoccardoGallouet,StampacchiaElliptic}.

We start by introducing the main elements and notations of our problem. We will then provide a detailed outline of the paper. Throughout this paper, $\Omega$ will be an open subset of $\R^N$ with finite Lebesgue measure.

Let $X=\left\{X_1, \ldots, X_m\right\}$ be a family of locally Lipschitz-continuous vector fields in $\mathbb{R}^N$ (with $N\geq2$) with coordinate representation $X_j=\left(c_{j 1}, \ldots, c_{j N}\right), j=1, \ldots, m$. That is,
$$
X_j=\sum_{k=1}^N c_{j k} \partial_{x_k},
$$
for some Lipschitz-continuous functions $c_{j k}$.
Moreover, for any $u\in C^1_c(\Omega)$, we will use \say{$X u$} to denote the $X$-gradient of $u$, which is
$$
X u=\left(X_1 u, \ldots, X_m u\right) = C(x)\nabla u,
$$
where $C(x)$ is the $m\times N$ matrix given by $C(x)=\{c_{j k}(x)\}$.

\begin{remark}
    For any function $u\in\elleom1$, the $X$-gradient is defined in the following distributional sense:
    $$\langle X_iu, \varphi\rangle = - \io u X_i\varphi\quad\forall \varphi\in C^1_c(\Omega)\quad i=1\ldots m.$$
    We also let $X^*$ denote the (formal) adjoint of the differential operator $X$.
\end{remark}

We are interested in existence and regularity of solutions to
\begin{equation}\label{Problemabase}
    \begin{dcases}
        X^*\left(A(x) X u\right) = f(x) \quad & \Omega\\
        u=0 \quad & \partial\Omega
    \end{dcases}
\end{equation}
where $\Omega$ is an open and bounded subset of $\R^N$, $f\in\elleom1$ and $A$ is a matrix-valued measurable function for which there exists $0<\alpha<\beta$ such that
\begin{equation}\label{HPsuA}
    \alpha\abs\eta^2\leq A(x)\eta\cdot\eta\leq\beta\abs\eta^2\quad \forall\eta\in\R^m.
\end{equation}

\begin{remark}
    No assumption on the regularity of the matrix $A(x)$, which will simply be measurable, has been made.
\end{remark}

By the Gauss-Green formula, we have, for any $u,v\in C^1_c(\Omega)$
     \begin{equation}         
     \begin{split}
         \langle X^*\left(A(x) X u\right), v \rangle &= A(x)Xu\cdot Xv
         \\&=
         A(x)C(x)\nabla u \cdot C(x)\nabla v= C(x)^TA(x)C(x)\nabla u \cdot \nabla v.
     \end{split}
     \end{equation}
    It follows that, given the $N\times N$ matrix $\tilde A(x)\colon=C(x)^TA(x)C(x)$ (which may not be uniformly elliptic), the operator $L=X^*\left(A(x)X\cdot\right)$ can be written as
    \begin{equation}\label{L_e_divergenza}
        Lu=-\operatorname{div}(\tilde A(x)\nabla u),
    \end{equation}
this suggests the following definition.
\begin{dhef}
    Given $f\in\elleom1$, we say that $u\in\SobomX1$ is a distributional solution to \eqref{Problemabase} if
$$\io A(x)Xu\cdot X\varphi=\io f(x)\varphi\quad\forall\varphi\in C^1_c(\Omega).$$
\end{dhef}

\begin{remark}
    Since the operator $L$ in \eqref{L_e_divergenza} is an elliptic operator in divergence form, it is natural to investigate whether the $L^p$ regularity results which hold in the uniformly elliptic case \cite{BoccardoGallouet,StampacchiaElliptic} to this class of possibly degenerate operators in divergence form. Note that, in general, $m$ may be strictly smaller than $N$, which implies that $\tilde A$ may not be a uniformly elliptic matrix. 
\end{remark}

Note that, under assumption \eqref{HPsuA}, we have that
\begin{equation}\label{X-elliptic-pointwise}
    \alpha \sum_{j=1}^m\abs{c_j(x)\cdot\xi}^2 \leq 
    \sum_{i, j=1}^N \tilde A_{i j}(x) \xi_i \xi_j
    \leq \beta \sum_{j=1}^m\abs{c_j(x)\cdot\xi}^2 \quad \forall x, \xi \in \mathbb{R}^N,
\end{equation}
This leads to the following coercivity property, which highlights the similarity with the elliptic case
\begin{equation}\label{X-elliptic}
    \alpha\io \abs{Xu}^2\leq \io \tilde A(x)\nabla u\cdot\nabla u\leq \beta\io \abs{Xu}^2\quad \forall u\in C^1_c(\Omega).
\end{equation}

\begin{remark}
    This property of the operator $L$ is called, in literature, $X$-ellipticity. We will equivalently use this term to refer to the inequality
    \begin{equation}
        \alpha\io \abs{Xu}^2\leq \io A(x)X u\cdot X u,
    \end{equation}
    which is crucial for our purposes.
\end{remark}

\section*{Outline}

In Section \ref{Sect_Geometric assumptions} we recall the standard assumptions on the vector fields $X$ and the properties of the functional spaces involved. A crucial role will be played by the Sobolev inequality, which is the main ingredient to establish $L^p$ regularity results for elliptic equations.

In Section \ref{Sect_existence}, we will prove the existence of a unique solution (found by approximation) to \eqref{Problemabase} with $\elle1$ data and in Section \ref{Sect_LpRegularity} we will prove that, as the summability of the datum $f$ increases, the summability of both the solution and its gradient will improve accordingly.

Sections \ref{Sect_DualitySol} and \ref{Sect_Measuredata} are devoted to the notion of duality solutions, a definition of solution (equivalent to the one of solution found by approximation) which exploits the linearity of the problem, and to the study of $X$-elliptic equations with measure data.

In the last sections, we state stability results for the solutions (which easily descend from the previous theorems) and compare the regularity estimates obtained with the classical elliptic estimates. To this end, we will apply our results to the study of linear equations involving the Heisenberg Laplacian.

\section{Assumptions on the vector fields and Functional spaces}\label{Sect_Geometric assumptions}

We start by recalling the notion of control distance (or Carnot-Carath\'eodory metric) associated to the family of vector fields $X$.

An absolutely continuous path $\gamma:[0, T] \rightarrow \Omega$ is said $X$-subunit if $\dot{\gamma}(t)=\sum_{j=1}^m \theta_j(t) X_j(\gamma(t))$, with $\sum_{j=1}^m \theta_j^2(t) \leq 1$, for almost every $t \in[0, T]$. Assuming that $\Omega$ is $X$-connected, i.e. for every $x, y \in \Omega$ there exists at least one $X$-subunit path connecting $x$ and $y$, we define 
$$d_X(x, y):=\inf \left\{T>0 \mid \exists \gamma:[0, T] \rightarrow \R^N\,  \text{$X$-subunit such that } \gamma(0)=x, \gamma(T)=y\right\}.$$
Under these connectedness assumption, $d_X$ is a metric on $\Omega$. It can be proved that $d_X(x, y) \rightarrow 0$ implies $|y-x| \rightarrow 0$, where $|\cdot|$ is the euclidean norm.

We will make the following assumptions on family of the vector fields $X$, which have been thoroughly outlined in \cite{Mazzoni-Green}, \cite{KL-Liouville} and \cite{HP-SobPoinc}.
\begin{itemize}
    \item The Carnot-Carath\'eodory metric $d_X$ is equivalent to the euclidean norm 
    $$|y-x| \rightarrow 0 \quad \Rightarrow\quad d_X(x, y) \rightarrow 0.$$ This is true, for example, if the vector fields $X_j$ are smooth and satisfy the H\"ormander condition, which means that the commutators of the vector fields $X_i$ up to some step $s\leq N$ span $\R^N$.
    \item $(\Omega,d_X,\mathcal L^N)$, where $\mathcal L^N$ denotes the Lebesgue measure, is a doubling space, which means that there exists a constant $Q>0$ such that 
    \begin{equation}\label{QDoubling}
        0<\mathcal{L}^N(B_{2r})\leq2^Q\mathcal{L}^N(B_{r}),
    \end{equation}
    where $B_r$ is the $d_X$-ball of radius $r$ . Note that we can assume, without loss of generality, that $Q>2$.
    \item there exist positive constants $C_P, \nu$ such that the following Poincar\'e inequality holds
    $$
    \fint_{B_r}\left|u-u_r\right| \mathrm{d} x \leq C_P r \fint_{B_{\nu r}}|X u| \mathrm{d} x, \quad \forall u \in C^1\left(\overline{B_{\nu r}}\right)
    $$
    for any $d_X$-ball $B_r$, with $u_r=:=\frac{1}{\left|B_r\right|} \int_{B_r} u\, \mathrm{d} x$.
\end{itemize}

\begin{ex}
    Aside from $X_i=\frac{\partial}{\partial x_i}$, which lead to the uniformly elliptic case, an example of family of vector fields satisfying these assumptions is given by
    $$
    \left\{\begin{aligned}
    X_i & =\frac{\partial}{\partial x_i}+2 y_i \frac{\partial}{\partial t},\quad i=1\ldots n \\
    Y_i & =\frac{\partial}{\partial y_i}-2 x_i \frac{\partial}{\partial t},\quad i=1\ldots n \\
    T & =\frac{\partial}{\partial t}
    \end{aligned}\right.
    $$
    on $\R^N$ with $N={2n+1}$. This structure, which is related to the Heisenberg group, will be studied more in detail in Section \ref{Sect_Heisenberg}.
\end{ex}

Since the existence of a solution to \eqref{Problemabase} is related to the coercivity of the bilinear form $$B(u,v)=\io A(x)X u \cdot X v,$$ we need a suitable class of Sobolev spaces. Here, we follow \cite{HP-SobPoinc}.
\begin{dhef}
    We define the space $W^{1,p}_0(\Omega,X)$ as the closure of $C^\infty_c(\Omega)$ under the norm
    $$\norma{u}{W^{1,p}_0(\Omega,X)}=\norma{Xu}{\elleom p}.$$
\end{dhef}
Note that, as in the uniformly elliptic case (when $X$ is the standard gradient), $W^{1,2}_0(\Omega,X)$ is a Hilbert space with the scalar product $(u,v)=\io Xu \cdot Xv$.
Moreover, $\SobomX p$ satisfies the following equivalent of the Sobolev embedding theorem (see \cite{KL-Liouville} and \cite{HP-SobPoinc}).
\begin{theo}(Sobolev Inequality). Let $1\leq p<Q$, with $Q$ given by \eqref{QDoubling}. There exists a positive constant $\mathcal{S}_p$ (only depending on $C_P$, $\nu$, $Q$ and $p$) such that
$$
\norma{u}{\elleom {p^*}}\leq \mathcal{S}_p\norma{Xu}{\elleom p} \quad \forall u \in W^{1,p}_0(\Omega,X),\quad p^*=\frac{p Q}{Q-p} .
$$
\end{theo}

\begin{remark}
    We use the notation $p^*$, as in the classical case, to denote the Sobolev exponent. We also define $p_*$ as the H\"older conjugate of $p^*$, namely $p_*=\frac{pQ}{Q(p-1)+p}$. 
\end{remark}

\begin{remark}\label{2starinbasso}
    $\elleom{2_*}$, with $2_*=\frac{2Q}{Q+2}$, is a subspace of $\left(W^{1,2}_0(\Omega,X)\right)^\prime$.
    Moreover, by definition of $X^*$, we know that
    $$\langle X^*(F(x)),v\rangle=\io C(x)^T F(x)\cdot\nabla v = \io F(x)\cdot Xv\quad\forall v\in C^1_c,$$
    which implies that $X^*(F)$ belongs to the dual of $\SobomX2$ if and only if $\abs{F}\in\elleom2$. 
\end{remark}

The ability to describe the elements of the dual of $\SobomX2$ allows us to give the following definition.
\begin{dhef}
    Given $f\in\elleom{\frac{2Q}{Q+2}}$ and $F\in\left(\elleom2\right)^m$, we say that $u\in\SobomX2$ is a weak solution to    
    \begin{equation}
        \begin{dcases}
            X^*\left(A(x) X u\right) = f(x)+X^*F(x) \quad & \Omega\\
            u=0 \quad & \partial\Omega
        \end{dcases}
    \end{equation}
    if
    $$\io A(x)Xu\cdot X v=\io f(x) v + \io F(x)\cdot Xv\quad\forall v\in \SobomX2.$$
\end{dhef}
\noindent Note that, by Lax Milgram's theorem, such a solution exists for any $f$ and $F$ as above.

We now recall other important properties of $\SobomX p$. For a more thorough discussion, we refer the reader to \cite{HP-SobPoinc}, \cite{KL-Liouville} and \cite{Mazzoni-Green}.

\begin{theo}
    $\SobomX p$ is a reflexive Banach space for $1<p<\infty$.
\end{theo}

\begin{theo}
    If $1\leq p<Q$ and $q<\frac{Qp}{Q-p}$, any bounded sequence $(u_n)_n\subset\SobomX p$ strongly converges, up to subsequences, in $\elleom q$.
\end{theo}

\begin{lemma}
    If $\psi$ is a Lipschitz-continuous function such that $\psi(0)=0$ and $u\in\SobomX p$, then $\psi(u)\in\SobomX p$ and $X\psi(u)=\psi'(u)Xu$.
\end{lemma}

\begin{remark}
    Under the assumptions stated in Section \ref{Sect_Geometric assumptions}, the space $\SobomX p$ can equivalently be defined via the Newtonian Sobolev spaces introduced by Shanmugalingam in \cite{Shanmu-Spaces}. For a detailed comparison of the various notions of Sobolev spaces, we refer the reader to \cite[Section 10.5]{Sobmetric-book} and the recent survey \cite{DG-gromov}.
    
    The Newtonian approach has the advantage of not requiring, in principle, neither a Poincar\'e inequality nor the doubling condition. We chose to require both properties because we are interested in the regularization properties that arise from an increase in summability in the datum of an elliptic equation, which is strongly related to the validity of a Sobolev inequality.
\end{remark}

We will also use Marcinkiewicz (see \cite{BoccardoCroce} for further details), of which we recall the definition and main properties.

\begin{dhef}
    Let $1\leq p<\infty$, we define the Marcinkiewicz space $\Mar p$ as the space of measurable functions from $\Omega$ to $\R$ for which there exists a constant $C_u\geq 0$ such that
    \begin{equation}\label{defMarcSpaces}
        \abs{\{\abs{u}\geq \lambda\}}\leq \frac{C_u}{\lambda^p}\quad\forall\,\lambda>0.
    \end{equation}
\end{dhef}

If we let $C_u$ denote the optimal constant in \eqref{defMarcSpaces}, we get $\Mar p$ is a Banach space with the norm $\norma{u}{\Mar{p}}:= C_u^\frac{1}{p}$. Furthermore, note that, since $\Omega$ is bounded, 
$$\elleom p\subset \Mar p\subset \elleom q\quad\forall\,q\in[1,p).$$
The following proposition characterizes the dual of Marcinkiewicz spaces using Lorentz spaces (denoted by $\elle {p,q}$). We refer the reader to \cite{Lorentz_ref} for further details.
\begin{prop}\label{prop:dualityMar}
    For every linear functional $\Lambda\in\left(\Mar p\right)'$ there exists a unique $v\in\elleom{p',1}$ such that the following integral representation
    $$\Lambda(u)=\io uv\quad \forall u\in\Mar p$$
    holds.
\end{prop}

\section{Existence of a solution with singular data}\label{Sect_existence}

To prove the existence of a distributional solution to \eqref{Problemabase} with datum $f\in\elleom1$ (often referred to as \say{singular data}), we will proceed by approximation. We begin by introducing the truncation function $T_k$ at height $k>0$:
\begin{equation}
        T_k(s) = \begin{dcases}
            k \quad & s> k\\
            s \quad & \abs s\leq k\\
            -k \quad & s<- k
        \end{dcases}
\end{equation}
and define $G_k(s)=s-T_k(s)=(\abs s-k)_+\rm{sgn}(s)$.

Given $f\in\elleom1$, we consider the approximating problem
\begin{equation}\label{Problemabase_approx}
    \left\{\begin{array}{cl}
    X^*\left(A(x) X u_n\right)=f_n(x) & \text { in } \Omega, \\
    u_n=0 & \text { on } \partial \Omega .
    \end{array}\right. 
\end{equation}
where $f_n(x)=T_n(f(x))$. We already know that, for every $n\in\N$, there exists a unique solution $u_n$ to \eqref{Problemabase_approx}. However, since $\elleom1$ is strictly larger than $\elleom {\frac{2Q}{Q+2}}$ (recall that $Q>2$), we cannot expect this sequence to be bounded in the energy space $\SobomX 2$.

The next result addresses a priori estimates a suitable Marcinkiewicz space, following the approach outlined in \cite{LavoroB6}.

\begin{theo}\label{March_estimates}
    The sequence $(u_n)_n$ is bounded in $\Mar{\frac{Q}{Q-2}}$ and the sequence of $X$-gradients $(Xu_n)_n$ is bounded in $\Mar{\frac{Q}{Q-1}}$.
\end{theo}
\begin{proof}
    Let $k>0$ and choose $T_k(u_n)$ as a test function in \eqref{Problemabase_approx}. Since $T_k'(s)=\rchi_{\{\abs s\leq k\}}$, we obtain
    $$\io A(x)X T_k(u_n)\cdot X T_k(u_n)=
    \io A(x)X u_n\cdot X T_k(u_n)=\io f_n(x)T_k(u_n)$$
    By $X$-ellipticity and the Sobolev inequality, we get
    $$\frac{\alpha}{\mathcal{S}_2^2}\left(\io \abs{T_k(u_n)}^{\frac{2Q}{Q-2}}\right)^{\frac{Q-2}{Q}}
    \leq 
    \alpha\io \abs{X T_k(u_n)}^2
    \leq 
    k\norma{f}{\elleom1}.$$
    Thus, narrowing the integration domain on the left hand side, we obtain
    $$k^2\frac{\alpha}{\mathcal S^2}\abs{\{\abs{u_n}\geq k\}}^\frac{Q-2}{Q}
    \leq
    \frac{\alpha}{\mathcal{S}_2^2}\left(\int_{\abs{u_n}\geq k} \abs{T_k(u_n)}^{\frac{2Q}{Q-2}}\right)^{\frac{Q-2}{Q}}
    \leq
    k\norma{f}{\elleom1}.
    $$
    Hence,
    $$\abs{\{\abs{u_n}\geq k\}}\leq C\frac{\norma{f}{\elleom1}^{\frac{Q}{Q-2}}}{k^{\frac{Q}{Q-2}}}.$$
    Let $t>0$. We want to estimate $\abs{\{\abs{Xu_n}\geq t\}}$. To that end, let $k>0$ (to be chosen later) and note that
    $$\abs{\{\abs{Xu_n}\geq t\}}
    \leq 
    \abs{\{\abs{Xu_n}\geq t\,,\,\abs{u_n}\leq k\}}
    +\abs{\{\abs{u_n}\geq k\}}.$$
    Now let us choose $T_k(u_n)$ as a test function in \eqref{Problemabase_approx} to obtain
    $$\alpha\io \abs{X u_n}^2\rchi_{\{\abs{u_n}\leq k\}}
    \leq 
    \alpha\io \abs{X T_k(u_n)}^2
    \leq 
    k\norma{f}{\elleom1}.$$
    Thus, narrowing the integration domain on the left hand side, we obtain
    $$\alpha t^2 \abs{\{\abs{Xu_n}\geq t\,,\,\abs{u_n}\leq k\}}
    \leq 
    \alpha\io \abs{X u_n}^2\rchi_{\{\abs{Xu_n}\geq t\,,\,\abs{u_n}\leq k\}}
    \leq 
    k\norma{f}{\elleom1}.$$
    Putting everything together, we obtain
    $$\abs{\{\abs{Xu_n}\geq t\}}
    \leq
    \alpha\frac{k}{t^2}\norma{f}{\elleom1}+C\frac{\norma{f}{\elleom1}^{\frac{Q}{Q-2}}}{k^{\frac{Q}{Q-2}}},
    $$
    which, minimizing over $k$, leads to
    $$
    \abs{\{\abs{Xu_n}\geq t\}}
    \leq
    C\frac{\norma{f}{\elleom1}^{\frac{Q}{Q-1}}}{t^{\frac{Q}{Q-1}}},
    $$
    thereby concluding the proof.
\end{proof}

\Cref{March_estimates} implies the convergence, up to a subsequence, of $u_n$ to a function $u$, both weakly in $\SobomX q$ for every $q<\frac{Q}{Q-1}$ and strongly in $\elleom s$ for every $s<\frac{Q}{Q-2}$. It follows that, given $\varphi\in C^1_c(\Omega)$, we can pass to the limit as $n\to\infty$ in 
$$\io A(x)X u_n X \varphi = \io f_n(x)\varphi$$
to prove that $u$ is a distributional solution to \eqref{Problemabase}.





\begin{remark}\label{Remark:measurebyapprox}
    The same approach can be used to prove the existence of a distributional solution to the problem $$Lu=\mu$$
    where $\mu$ is a Radon measure on $\Omega$. In this case, one would choose a sequence $(f_n)_n\subset\elleom1$ that converges to $\mu$ with respect to the weak-$*$ topology of measures (note that such sequence can be chosen, without loss of generality, to have total mass equal to the one of $\mu$).
    We will discuss the properties of solutions to equations with measure data in Section \ref{Sect_Measuredata}.
\end{remark}

Note that the limit point $u$ of $(u_n)_n$ may, in principle, not be unique. Indeed, the works by Serrin \cite{Serrin1964} and Prignet \cite{Prignet2007RemarksOE} show that there may be infinitely many distributional solutions (with infinite energy) to an elliptic equation.
The next result, however, shows the solution found by approximation is unique. This is the main reason to consider the solution found by approximation as the \say{correct one} when the definition of weak solution is not available.

\begin{theo}\label{Thm:uniqueapprox}
    Let $f\in\elleom1$, then the solution $u$ found by approximation is unique and does not depend on the approximating sequence $(f_n)_n$.
\end{theo}

\begin{proof}
    Let $n,m\in\N$. By linearity, $u_n-u_m$ solves the problem
    \begin{equation}
        X^*(A(x)X(u_n-u_m))=f_n(x)-f_m(x)        
    \end{equation}
    with zero Dirichlet boundary conditions. We can thus invoke Theorem \ref{March_estimates} to conclude that
    $$\norma{u_n-u_m}{\Mar {\frac{Q}{Q-2}}}\leq C\norma{f_n-f_m}{\elleom1}.$$
    Since the sequence $(f_n)_n$ is Cauchy in $\elleom1$, the sequence $(u_n)_n$ is also Cauchy in $\Mar{\frac{Q}{Q-2}}$, which implies the uniqueness of the solution found by approximation (for any given sequence $(f_n)_n$ which approximates $f$).

    Now let $(f_n)_n$ and $(g_n)_n$ be two sequences of $\elleom {\frac{2Q}{Q+2}}$ functions which converge to $f$ in $\elleom1$. Let $u_n$ and $v_n$ be the corresponding solutions to \eqref{Problemabase_approx}. By linearity, $u_n-v_n$ solves the problem
    \begin{equation}
        X^*(A(x)X(u_n-v_n))=f_n(x)-g_n(x)        
    \end{equation}
    with zero Dirichlet boundary conditions. We can thus apply  \Cref{March_estimates} again to conclude that
    $$\norma{u_n-v_n}{\Mar {\frac{Q}{Q-2}}}\leq C\norma{f_n-g_n}{\elleom1}.$$
    Since the term on the right goes to zero as $n\to\infty$, the sequences $(u_n)_n$ and $(v_n)_n$ have the same limit.
\end{proof}

Since \Cref{Thm:uniqueapprox} provides a notion of solution that applies to any datum belonging, at least, to $\elleom1$. We will, in what follows, refer to a function solving \eqref{Problemabase}, whether in the weak sense or obtained via approximation, simply as a \say{solution}. The precise meaning should be understood in relation to the regularity of the given datum.

\section{Regularity of the solutions}\label{Sect_LpRegularity}

In this section, we investigate the regularity properties of the solution to \eqref{Problemabase}. More precisely, we prove that an increase in summability of the datum $f$ of \eqref{Problemabase} leads to an increase in summability of the solution $u$.

Our first result consider data with high summability. To prove it, we will use the following lemma from \cite{StampacchiaElliptic}.
\begin{lemma}\label{LemmaStampacchiaLinfty}
    Let $\psi: \mathbb{R}^{+} \rightarrow \mathbb{R}^{+}$be a nonincreasing function such that
    $$
    \psi(h) \leq \frac{M \psi(k)^\delta}{(h-k)^\gamma}, \quad \forall\, h>k \geq 0,
    $$
    for some $M>0, \delta>1$ and $\gamma>0$. Then $\psi$ has a zero $d\in\R$, where
    $$
    d^\gamma=M \psi(0)^{\delta-1} 2^{\frac{\delta \gamma}{\delta-1}} .
    $$
\end{lemma}

\begin{theo}\label{Stampacchia-Linfty}
    Let $f\in\elleom p$, with $p>\frac{Q}{2}$. 
    Then the solution $u$ to \eqref{Problemabase} belongs to $\elleom\infty$, and there exists a constant $C$, which depends only on $Q, \Omega, p$ and $\alpha$, such that
    $$
    \norma{u}{\elleom\infty} \leq C \norma f{\elleom p}.
    $$
\end{theo}

\begin{remark}
    Note that, assuming $Q>2$, $\frac{Q}{2}$ is strictly greater than $\frac{2Q}{Q+2}$, which implies that $u$ should be intended as a solution in the weak sense.
\end{remark}

\begin{proof} Here, we follow \cite{StampacchiaElliptic}.
    Let $k \geq 0$ and choose $v=G_k(u)$ as test function in \eqref{Problemabase}. We now define $A_k=\{x \in \Omega:|u(x)| \geq k\}$. 
    Computing the $X$-gradient of $G_k(u)$ we get
    $$
    \io A(x)  Xu \cdot  Xu \rchi_{A_k} = \io f G_k(u).
    $$
    We now recall that $G_k(u)=0$ outside of $A_k$, so that
    $$
    \alpha \int_{A_k}\left| X G_k(u)\right|^2 
    \leq 
    \io A(x)  Xu \cdot  Xu \rchi_{A_k}
    =
    \int_{A_k} f G_k(u).
    $$
    Using the Sobolev inequality (on the left-hand side), and the H\"older inequality (on the right-hand side), one has
    $$
    \frac{\alpha}{\mathcal{S}_2^2}\left(\int_{A_k}\left|G_k(u)\right|^{2^*}\right)^{\frac{2}{2^ *}} \leq\left(\int_{A_k}\abs{f}^{\frac{2Q}{Q+2}}\right)^{\frac{1}{2_*}} \left(\int_{A_k}\left|G_k(u)\right|^{2^*}\right)^{\frac{1}{2^*}} .
    $$
    Simplifying equal terms (note that if $G_k$ was zero and the simplification was not allowed, we would already know that $\abs{u}\leq k$ and the proof would be complete), we have
    $$
    \int_{A_k}\left|G_k(u)\right|^{2^*} \leq\left(\frac{\mathcal{S}_2^2}{\alpha}\right)^{2^*}\left(\int_{A_k}\abs{f}^{2_ *}\right)^{\frac{2^*}{2_ *}} .
    $$
    Since we know that $f$ belongs to $\elleom p$, with $p>\frac{Q}{2}\geq 2_*=\frac{2Q}{Q+2}$, we can use H\"older's inequality once again to get
    $$
    \int_{A_k}\left|G_k(u)\right|^{2^*} \leq\left(\frac{\mathcal{S}_2^2\|f\|_{\elleom p}}{\alpha}\right)^{2^*} \abs{A_k}^{\frac{2^*}{2_*}-\frac{2^*}{p}} .
    $$
    
    We now want to apply \Cref{LemmaStampacchiaLinfty} to $\psi(k)=\abs{A_k}$: take $h>k$, so that $A_h \subseteq A_k$, and $G_k(u) \geq h-k$ on $A_h$. Thus,
    $$
    (h-k)^{2^*} \abs{A_h} \leq\left(\frac{\mathcal{S}_2^2\|f\|_{\elleom p}}{\alpha}\right)^{2^*} \abs{A_k}^{\frac{2^*}{2_*}-\frac{2^*}{p}},
    $$
    which implies
    $$
    \abs{A_h} \leq\left(\frac{\mathcal{S}_2^2\|f\|_{\elleom p}}{\alpha}\right)^{2^*} \frac{\abs{A_k}^{\frac{2^*}{2_*}-\frac{2^*}{p}}}{(h-k)^{2^*}} .
    $$
    Naming 
    $$
    M=\left(\frac{\mathcal{S}_2^2\|f\|_{\elleom p}}{\alpha}\right)^{2^*}, \quad \gamma=2^* \quad \delta=\frac{2^*}{2_*}-\frac{2^*}{p}=\frac{2p+pQ-2Q}{pQ-2Q}
    $$
    and observing that the assumption $p>\frac Q2$ is equivalent to $\delta>1$, we get that
    $\psi(k)=\abs{A_k}$ satisfies the assumptions of the previous lemma.
    We thus get that $\psi(d)=0$, where
    $$
    d^{2^*}=C(\Omega, Q, p, \alpha) M .
    $$
    Since $\abs{A_d}=0$, we have $|u| \leq d$ almost everywhere, which implies
    $$
    \|u\|_{\elle{\infty}(\Omega)} \leq d=C(Q, \Omega, p, \alpha)\|f\|_{\elleom p},
    $$
    which concludes the proof.
\end{proof}

We now prove that, as in the elliptic case, the summability of $f$ directly affects the summability of the solution.

\begin{theo}\label{Stampacchia-p**-grande}
    Let $f$ belong to $\elleom p$, with $2_*=\frac{2Q}{Q+2} \leq p<\frac{Q}{2}$. Then the solution $u$ to \eqref{Problemabase} belongs to $\elleom {p^{* *}}$, with $p^{* *}=\frac{pQ}{Q-2 p}$, and there exists a constant $C$, only depending on $Q, \Omega, m$ and $\alpha$, such that
    $$
    \norma u{\elleom{p^{* *}}}\leq C\norma f{\elleom p}.
    $$
\end{theo}

\begin{proof} Here we follow \cite{BoccardoGallouet}.
    We begin by observing that if $p=2_*$, then $p^{* *}=2^*$, so that the result is true in this limit case by the Sobolev embedding. Therefore, we only have to deal with the case $p>2_*$.
    
    Let $k>0$, $\gamma>1$ and choose $v=\left|T_k(u)\right|^{2 \gamma-2} T_k(u)$ as test function in \eqref{Problemabase_approx}. Computing the $X$-gradient, we get
    $$
    (2 \gamma-1) \io A(x)  Xu \cdot X T_k(u)\left|T_k(u)\right|^{2 \gamma-2}
    =
    \io f(x)\left|T_k(u)\right|^{2 \gamma-2} T_k(u) .
    $$
    We now observe that $X u=X T_k(u)$ where $X T_k(u) \neq 0$ and use the $X$-ellipticity to obtain
    $$
    \alpha(2 \gamma-1) \io\abs{X T_k(u)}^2\left|T_k(u)\right|^{2 \gamma-2} \leq \io\abs{f(x)}\left|T_k(u)\right|^{2 \gamma-1} .
    $$
    Since $\left|X T_k(u)\right|^2\left|T_k(u)\right|^{2 \gamma-2}=\frac{1}{\gamma^2}\left|X( T_k(u)^\gamma)\right|^2$, we have
    $$
    \frac{\alpha(2 \gamma-1)}{\gamma^2} \io\left|X( T_k(u)^\gamma)\right|^2 \leq \io\abs{f(x)}\left|T_k(u)\right|^{2 \gamma-1} .
    $$
    Using Sobolev inequality and H\"older inequality, we obtain
    $$
    \frac{\alpha(2 \gamma-1)}{\mathcal{S}_2^2 \gamma^2}\left(\io\left|T_k(u)\right|^{\gamma2^*}\right)^{\frac{2}{2^*}} 
    \leq
    \norma f{\elleom p}\left(\io\left|T_k(u)\right|^{(2 \gamma-1) p^{\prime}}\right)^{\frac{1}{p^{\prime}}} .
    $$
    We now choose $\gamma$ so that $\gamma 2^*=(2 \gamma-1) p^{\prime}$, that is $\gamma=\frac{p^{* *}}{2^*}$. 
    Observe that, with this choice, $\gamma>1$ if and only if $p>2_*$ (which is true). 
    Since $p<\frac{Q}{2}$, we also have $\frac{2}{2^*}>\frac{1}{p^{\prime}}$, and so
    $$
    \left(\io\left|T_k(u)\right|^{p^{* *}}\right)^{\frac{2}{2^*}-\frac{1}{p'}} \leq C(Q, \Omega, p, \alpha)\norma f{\elleom p}.
    $$
    It now suffice to observe that $\frac{2}{2^*}-\frac{1}{p'}=\frac1{p^{* *}}$ and apply Fatou's lemma as $k\to+\infty$ to get
    $$
    \norma u{\elle{p^{* *}}}\leq C(Q, \Omega, p, \alpha)\norma f{\elleom p},
    $$which concludes the proof.
\end{proof}

Now we prove an analogous result in the case $p\in(1,2_*)$.

\begin{theo}\label{Stampacchia-p**-piccolo}
    Let $f$ belong to $\elleom p$, with $1 < p<2_*$. Then the solution $u$ to \eqref{Problemabase} found by approximation belongs to $\elleom {p^{* *}}$, with $p^{* *}=\frac{pQ}{Q-2 p}$, and there exists a constant $C$, only depending on $Q, \Omega, p$ and $\alpha$, such that
    $$
    \norma {u}{\elleom{p^{* *}}}\leq C\norma f{\elleom p}.
    $$
\end{theo}

\begin{proof}
    Let $k>0$, $\varepsilon>0$, $\gamma>\frac{1}{2}$ and choose $v=sgn(u_n)\left(\left(\abs{T_k(u_n)}+\varepsilon\right)^{2 \gamma-1}-\varepsilon^{2\gamma-1}\right)$ as test function in \eqref{Problemabase_approx}. Computing the $X$-gradient leads to
    $$
    (2 \gamma-1) \io A(x)  Xu \cdot X T_k(u_n)\left|T_k(u_n)+\varepsilon\right|^{2 \gamma-2}
    \leq \io\abs{f_n(x)}\left(\abs{T_k(u_n)}+\varepsilon\right)^{2 \gamma-1} .
    $$
    We now observe that $X u=X T_k(u_n)$ where $X T_k(u_n) \neq 0$ and use the $X$-ellipticity to get
    \begin{equation}\label{StimaBase-p**}
        \alpha(2 \gamma-1) \io\abs{X T_k(u_n)}^2\left(\abs{T_k(u_n)}+\varepsilon\right)^{2 \gamma-2} \leq \io\abs{f_n(x)}\left(\abs{T_k(u_n)}+\varepsilon\right)^{2 \gamma-1} .
    \end{equation}
    Since $\abs{X T_k(u_n)}^2\left(\abs{T_k(u_n)}+\varepsilon\right)^{2 \gamma-2}=\frac{1}{\gamma^2}\left|X (\abs{T_k(u_n)}+\varepsilon)^\gamma\right|^2$, we have
    $$
    \frac{\alpha(2 \gamma-1)}{\gamma^2} \io\left|X (\abs{T_k(u_n)}+\varepsilon)^\gamma\right|^2 
    \leq 
    \io\abs{f_n(x)}\left(\abs{T_k(u_n)}+\varepsilon\right)^{2 \gamma-1} .
    $$
    Using Sobolev inequality and H\"older inequality, we obtain
    \begin{equation}
        \begin{split}     
    \frac{\alpha(2 \gamma-1)}{\mathcal{S}_2^2 \gamma^2}
    \bigg(\io&(\abs{T_k(u_n)}+\varepsilon)^{\gamma2^*}-\varepsilon^{\gamma2^*}\bigg)^{\frac{2}{2^*}} \\&
    \leq
    \norma f{\elleom p}\left(\io\left(\abs{T_k(u_n)}+\varepsilon\right)^{(2 \gamma-1) p^{\prime}}\right)^{\frac{1}{p^{\prime}}} .
        \end{split}
    \end{equation}
    We now choose $\gamma$ so that $\gamma 2^*=(2 \gamma-1) p^{\prime}$, that is $\gamma=\frac{p^{* *}}{2^*}$ and pass to the limit as $\varepsilon\to 0$ by Dominated Convergence Theorem (note that $\abs{T_k(u_n)}+\varepsilon\leq k+1$ for $\varepsilon$ small enough). 
    Observe that, with this choice, $\gamma>\frac{1}{2}$ if and only if $p>1$ (which is true). We also have $\frac{2}{2^*}>\frac{1}{p^{\prime}}$, and so
    $$
    \left(\io\left|T_k(u_n)\right|^{p^{* *}}\right)^{\frac{2}{2^*}-\frac{1}{p'}} \leq C(Q, \Omega, p, \alpha)\norma f{\elleom p}.
    $$
    It now suffice to observe that $\frac{2}{2^*}-\frac{1}{p'}=\frac1{p^{* *}}$ and apply Fatou's lemma as $k\to+\infty$, to get
    $$
    \norma {u_n}{\elle{p^{* *}}}\leq C(Q, \Omega, p, \alpha)\norma f{\elleom p}.
    $$
    Applying Fatou's Lemma again as $n\to\infty$, we thus have that
    $$
    \norma {u}{\elle{p^{* *}}}\leq C(Q, \Omega, p, \alpha)\norma f{\elleom p},
    $$
    which concludes the proof.
\end{proof}

\begin{remark}
    Note that, since there are counterexamples in the uniformly elliptic case, we cannot prove the same result for $p=1$. Indeed, the regularity obtained in \Cref{March_estimates} is the best that can be attained for a general $\elle1$ datum.
\end{remark}

The next result concerns the regularity of the $X$-gradient of the solution $u$. Note that, without imposing further requirements on the matrix $A$, we should not expect $u$ to belong to a Sobolev space better than $\SobomX2$, as the regularity result by Meyers \cite{Meyers} (in the uniformly elliptic case) shows that such an improvement would depend on the ellipticity constant of $A$.
Nevertheless, in the case of infinite-energy solutions (that is, when the datum $f$ does not belong to $\elleom {2_*}$), a regularizing effect on the $X$-gradients of the solution is in place.

\begin{theo}\label{BG_gradiente}
    Let $f\in\elleom p$ with $1<p<2_*$. Then the solution $u$ to \eqref{Problemabase} found by approximation belongs to $\SobomX{p^*}$.
\end{theo}

\begin{proof}
    Let $q\in(1,2)$, we want to estimate
    $$\io \abs{X T_k(u_n)}^q.$$
    To do so, note that, taking $\epsilon>0$ and $\gamma$ as in the proof of \Cref{Stampacchia-p**-piccolo}, we have
    $$\io \abs{X T_k(u_n)}^q=\io \frac{\abs{X T_k(u_n)}^q}{\left(\abs{T_k(u_n)}+\varepsilon\right)^{q(1-\gamma)}}\left(\abs{T_k(u_n)}+\varepsilon\right)^{q(1-\gamma)},$$
    which, applying H\"older's inequality with exponents $\frac{2}{q}$ and $\frac{2}{2-q}$, yields
    $$\io \abs{X T_k(u_n)}^q\leq
    \left(\io \frac{\abs{X T_k(u_n)}^2}{\left(\abs{T_k(u_n)}+\varepsilon\right)^{2(1-\gamma)}}\right)^{\frac{q}{2}}
    \left(\io\left(\abs{T_k(u_n)}+\varepsilon\right)^{\frac{2q}{2-q}(1-\gamma)}\right)^{\frac{2-q}{2}}.$$
    Now observe that the first term on the right-hand side, by \eqref{StimaBase-p**} and the fact that $(2\gamma-1)p'=p^{**}$, is bounded by $C\left(\norma{f}{\elleom p}\norma{u_n+\varepsilon}{\elleom{p^{**}}}^\frac{1}{2\gamma-1}\right)^\frac{2}{q}$. On the other hand, the choice 
    $$\frac{2q}{2-q}(1-\gamma)=p^{**}$$
    leads to $q=m^*$ and thus, passing to the limit as $\varepsilon\to0$ and then as $k\to\infty$, we have
    $$\norma{Xu_n}{\elleom {p^*}}\leq C\norma{f}{\elleom p}$$
    for all $n\in\N$. It follows that, being $Xu$ the weak limit of $(Xu_n)_n$ in some Lebesgue space, the same result holds for $Xu$ by lower semicontinuity of the norm.
\end{proof}

We conclude this section with a result analogous to \Cref{Stampacchia-Linfty} and \Cref{Stampacchia-p**-grande}, which can be proved using the same techniques we previously employed.
\begin{theo}\label{Stampacchia_p*divergenza}
    Let $F\in(\elleom p)^m$ with $p\geq 2$. Then the problem
    \begin{equation}
        \begin{dcases}
            X^*\left(A(x) X u\right) = X^*F(x) \quad & \Omega\\
            u=0 \quad & \partial\Omega,
        \end{dcases}
    \end{equation}
    admits a unique weak solution $u\in\SobomX2$ such that
    \begin{itemize}
        \item if $p>Q$, then $u\in\elleom\infty$;
        \item if $2\leq p<Q$, then $u\in\elleom{\frac{Qp}{Q-p}}$.
    \end{itemize}
\end{theo}

\section{Duality solutions}\label{Sect_DualitySol}

Here we discuss a different technique to find solutions to \eqref{Problemabase}. This definition, which relies on the linearity of the problem, is less suited for generalizations to nonlinear equations than the one of solution found by approximation. Nevertheless, when both definitions are available, the two notions of solution coincide, as we will show in \Cref{Duality&Approx}.

Consider the solution $v$ to the adjoint problem
\begin{equation}\label{ProblemaAdj}
    \begin{dcases}
        X^*\left(A^T(x) X v\right) = g(x) \quad & \Omega\\
        v=0 \quad & \partial\Omega,
    \end{dcases}
\end{equation}
where $g\in\elleom\infty$, we have $v\in\elleom\infty$. Thus, choosing $u$ as a test function in \eqref{ProblemaAdj} and $v$ as a test function in \eqref{Problemabase}, we have
\begin{equation}\label{ideaduality}
    \io f(x)v = \io A(x)Xu\cdot Xv = \io Xu\cdot A^T(x)Xv = \io g(x)u.
\end{equation}
Note that the identity \eqref{ideaduality} only requires $u$ to belong to $\elleom1$. Following \cite{StampacchiaElliptic}, we give the following definition.
\begin{dhef}
    $u\in\elleom1$ is a duality solution to \eqref{Problemabase} if, for any $g\in\elleom\infty$, we have
    \begin{equation}\label{DualitySolDef}
        \io g(x)u=\io f(x)v,
    \end{equation}
    where $v\in\elleom\infty$ is the solution to \eqref{ProblemaAdj} with datum $g$.    
\end{dhef}

This definition, introduced by Stampacchia in \cite{StampacchiaElliptic}, has the advantage of not requiring any regularity, in principle, on the solution $u$. This is particularly relevant when $f$ has low summability, as the solution may have infinite energy (as shown by \Cref{BG_gradiente}).

\begin{lemma}
    Let $f\in\elleom1$, then there exists a unique duality $u$ solution to \eqref{Problemabase}. Furthermore, $u\in\elleom q\,\,\forall q<\frac{Q}{Q-2}$.
\end{lemma}

\begin{proof}
    Let $p>\frac{Q}{2}$ and define the linear functional $\Lambda: L^p(\Omega) \rightarrow \mathbb{R}$ as
    $$
    \Lambda(g)=\int_{\Omega} f v,
    $$
    where $v$ solves \eqref{ProblemaAdj} with datum $g$. 
    By \Cref{Stampacchia-Linfty}, the functional is well-defined and there exists $C>0$ such that
    $$
    |\Lambda(g)| \leq \int_{\Omega}|f||v| \leq\|f\|_{L^1(\Omega)}\|v\|_{L^{\infty}(\Omega)} \leq C\|f\|_{L^1(\Omega)}\|g\|_{L^p(\Omega)}.
    $$
    It follows that $T$ is continuous on $L^p(\Omega)$, which (by the Riesz representation theorem) implies that there exists a unique $u_p\in \elleom{p^{\prime}}$ such that
    $$
    \Lambda(g)=\int_{\Omega} u_p g, \quad \forall g \in L^p(\Omega) .
    $$
    Since $L^{\infty}(\Omega) \subset L^p(\Omega)$, we have
    $$
    \int_{\Omega} u_p g=\Lambda(g)=\int_{\Omega} f v, \quad \forall g \in L^{\infty}(\Omega),
    $$
    so that $u_p$ is a duality solution to \eqref{Problemabase}, as desired. Moreover, $u_p$ does not depend on $p$: to prove this, let $p>q>\frac{Q}{2}$. Then:
    $$
    \int_{\Omega} u_p g=\int_{\Omega} f v=\int_{\Omega} u_q g, \quad \forall g \in L^{\infty}(\Omega),
    $$
    so that $u_p=u_q$. Therefore, there exists a unique function $u$ which is a duality solution to \eqref{Problemabase} and belongs to $L^{p^{\prime}}(\Omega)$ for every $p>\frac{Q}{2}$. It follows that $u$ belongs to $L^q(\Omega)$ for every $q<\frac{Q}{Q-2}$, as desired.
\end{proof}

It turns out that (as in the elliptic case) the duality solution is exactly the solution found by approximation.

\begin{theo}\label{Duality&Approx}
    Let $u$ be the solution to \eqref{Problemabase} found by approximation. Then $u$ is also the duality solution to \eqref{Problemabase}.
\end{theo}

\begin{proof}
    Let $g\in\elleom\infty$ and choose the solution $v$ to \eqref{ProblemaAdj} with datum $g$ as a test function in \eqref{Problemabase_approx}. We have
    $$\io g(x)u_n=\io A(x)Xu_n Xv = \io f_n(x) v,$$
    which, passing to the limit as $n\to\infty$ (recall that the sequence $(u_n)_n$ is strongly convergent in $\elleom1$ to $u$) leads to
    $$\io g(x)u=\io f(x) v,$$
    which concludes the proof.
\end{proof}

We conclude this section by presenting a duality argument, suggested by L. Boccardo \cite{ConsigliLucio}, which improves Theorem \ref{Stampacchia-Linfty}. 
\begin{theo}
    If the datum $f$ in \eqref{Problemabase} belongs to $\elleom{\frac{Q}{2},1}$, then the solution $u$ belongs to $\elleom\infty$ and
    \begin{equation}
        \norma{u}{\elleom\infty}\leq C(Q,\Omega) \norma{f}{\elleom{\frac{Q}{2},1}}.
    \end{equation}
\end{theo}
\begin{proof}
    Let $g\in\elleom1$ and consider the solution $v$ of the adjoint problem \eqref{ProblemaAdj}. By \Cref{prop:dualityMar} and \Cref{March_estimates}, we have
    \begin{equation}
        \begin{split}
            \abs{\io g(x)u} = \abs{\io f(x)v} 
    &\leq C(Q,\Omega) \norma{v}{\Mar{\frac{Q}{Q-2}}}\norma{f}{\elleom{\frac{Q}{2},1}}
    \\&\leq C(Q,\Omega) \norma{g}{\elleom1}^\frac{Q}{Q-2}\norma{f}{\elleom{\frac{Q}{2},1}}.
        \end{split}
    \end{equation}    
    Taking the supremum over all $g\in\elleom1$ with $\norma{g}{\elleom1}=1$ leads to
    $$\norma{u}{\elleom\infty}\leq C(Q,\Omega) \norma{f}{\elleom{\frac{Q}{2},1}}.$$
\end{proof}

\section{Measure data}\label{Sect_Measuredata}

The definition of duality solutions and H\"older regularity results solutions to uniformly elliptic equations with highly summable data (De Giorgi - Nash - Moser theory) serve as a motivation to consider, by duality, elliptic equations with measure data. 
Such problems may also be considered natural from an approximation perspective, in the sense that the limit points of bounded sequences in $\elleom1$ should be sought in the space of Radon measures on $\Omega$.

In the $X$-elliptic case, the first H\"older regularity result was proved by Franchi and Lanconelli in \cite{FL-Holder}. More recently, Sawyer and Wheeden \cite{Subelliptic-DG} provided a H\"older regularity theorem under rather relaxed (although very close to necessary) conditions, of which our assumptions are a particular case.
More precisely, \cite[Theorem 7]{Subelliptic-DG} implies the following result.

\begin{theo}\label{Theo_HolderReg}
    Let $u$ be the unique solution to
    \begin{equation}
        \begin{dcases}
            Lu = f(x) + X^* F(x) \quad & \Omega\\
            u=0 \quad & \partial\Omega
        \end{dcases}
    \end{equation}
    with $f\in\elleom p$ with $p>\frac{Q}{2}$ and $F\in\left(\elleom q\right)^m$ with $q>Q$. Then for every compact subset $K\subset\Omega$ there exists $\alpha=\alpha(K)>0$ and $C=C(K,\norma{f}{\elleom p},\norma{F}{\elleom q},\norma{u}{\elleom2})$ such that
    $$\norma{u}{C^\alpha(K)}\leq C.$$
\end{theo}

\begin{remark}
    The requirement for the existence of an accumulating sequence of cutoff functions (in Theorem 7 of \cite{Subelliptic-DG}) can be satisfied choosing a sequence of radii $r_j=\frac{1}{j}r$ and using Theorem 10 of \cite{KL-Liouville} to construct a cutoff function $\psi_j$ such that
    $$\rchi_{B_{r_{j+1}}(x)}\leq \psi_j\leq\rchi_{B_{r_{j}}(x)}.$$
\end{remark}

\Cref{Theo_HolderReg} allows us to extend the previous existence and uniqueness results for duality solutions to any measure datum $\mu\in\mathcal M(\Omega)$, obtaining the following result.

\begin{theo}\label{Stampacchia_measuredata}
    For every $\mu\in\mathcal M (\Omega)$, there exists a unique duality solution $u$ of 
\begin{equation}\label{Problema_measure}
    \begin{dcases}
        X^*\left(A(x) X u\right) = \mu \quad & \Omega\\
        u=0 \quad & \partial\Omega.
    \end{dcases}
\end{equation}
Such solution is the one found by approximation, belongs to $\Mar{\frac{Q}{Q-2}}$ and its $X$-gradient belongs to $\Mar{\frac{Q}{Q-1}}$. Moreover, the following estimates hold:
    \begin{gather}
        \norma{u}{\Mar{\frac{Q}{Q-2}}}\leq C(Q,\Omega,\alpha)\norma{\mu}{\mathcal M(\Omega)}\\ 
        \norma{X u}{\Mar{\frac{Q}{Q-1}}}\leq C(Q,\Omega,\alpha)\norma{\mu}{\mathcal M(\Omega)}.
    \end{gather}
\end{theo}

We also point out that the uniqueness of the solution $u$ can be proved directly (working by approximation, see also Remark \ref{Remark:measurebyapprox}), without employing the notion of duality solution, as we show in the next result.

\begin{lemma}
    Given a measure datum $\mu\in\mathcal{M}(\Omega)$, the solution found by approximation $u$ does not depend on the approximating sequence $(f_n)_n$.
\end{lemma}

\begin{proof}
    Let $(f_n)_n$ and $(g_n)_n$ be sequences in $\elleom\infty$ such that both $f_n$ and $g_n$ converge to $\mu$ in the weak-* topology of measures, let $u_n$ and $v_n$ be the corresponding solutions to 
    \begin{equation}
        \begin{dcases}
            L\psi_n= X^* \left(\frac{X(u_n-v_n)}{\abs{X(u_n-v_n)}}\right)\quad &\Omega\\
            \psi_n=0\quad & \partial\Omega.
        \end{dcases}
    \end{equation}
    and let $u$, $v$ be their limit points as $n\to\infty$. Consider the solution $\psi_n$ of the problem
    \begin{equation}\label{Dualità_strana_UMI}
        \begin{dcases}
            L\psi_n= X^* \left(\frac{X(u_n-v_n)}{\abs{X(u_n-v_n)}}\right)\quad &\Omega\\
            \psi_n=0\quad & \partial\Omega.
        \end{dcases}
    \end{equation}
    Choosing $u_n-v_n$ as a test function in \eqref{Dualità_strana_UMI} leads to
        \begin{equation}\label{proof:approxunique1}
    \begin{split}
        \io A(x)X(u_n-v_n)\cdot X\psi_n &= \langle X^* \left(\frac{X(u_n-v_n)}{\abs{X(u_n-v_n)}}\right),u_n-v_n\rangle\\&= \io \frac{\abs{X(u_n-v_n)}^2}{\abs{X(u_n-v_n)}}=\io \abs{X(u_n-v_n)}.
    \end{split}
    \end{equation}
    At the same time, the left-hand side satisfies (by linearity of \eqref{Problemabase_approx})
    \begin{equation}\label{proof:approxunique2}
        \io A(x)X(u_n-v_n)\cdot X\psi_n=\io (f_n-g_n)\psi_n.
    \end{equation}
    Putting \eqref{proof:approxunique1} and \eqref{proof:approxunique2} together, we get
    \[
    \io \abs{X(u_n-v_n)}=\io (f_n-g_n)\psi_n.
    \]
    By \Cref{Theo_HolderReg}, the sequence $(\psi_n)_n$ is uniformly bounded and uniformly $\alpha$-H\"older continuous for some $\alpha>0$, which implies that it converges (up to subsequences) uniformly in $\Omega$. On the other hand, we know that $f_n-g_n\stackrel{\ast}{\rightharpoonup}0$. It follows that
    $$\io \abs{X(u_n-v_n)}\to 0\text{ as } n\to\infty,$$
    which concludes the proof.
\end{proof}

\section{Stability results}

The linearity of Equation \eqref{Problemabase}, along with with Theorems \ref{Stampacchia-Linfty}, \ref{Stampacchia-p**-grande}, \ref{Stampacchia-p**-piccolo}, \ref{March_estimates}, \ref{BG_gradiente}, \ref{Stampacchia_p*divergenza}, \ref{Stampacchia_measuredata} and Lax Milgram's Theorem, allows us to prove the following stability results.

\begin{cor}\label{Stability_Dati_fg}
    Let $f,g\in\elleom p$ with $p\geq1$ and let $u,v$ be the corresponding solutions to \eqref{Problemabase}. Then there exists a constant $C=C(Q, \Omega, p, \alpha)$ such that
    \begin{itemize}
        \item If $p>\frac{Q}{2}$, then $$\norma{u-v}{\elleom{\infty}}\leq C\norma{f-g}{\elleom p}$$ and $$\norma{u-v}{\SobomX2}\leq C\norma{f-g}{\elleom p};$$
        \item If $\frac{2Q}{Q+2}\leq p <\frac{Q}{2}$, then $$\norma{u-v}{\elleom{p^{**}}}\leq C\norma{f-g}{\elleom p}$$ and $$\norma{u-v}{\SobomX2}\leq C\norma{f-g}{\elleom p};$$
        \item If $1<p<\frac{2Q}{Q+2}$, then $$\norma{u-v}{\SobomX{p^{*}}}\leq C\norma{f-g}{\elleom p};$$
        \item If $p=1$, then $$\norma{u-v}{\Mar{1^{**}}}\leq C\norma{f-g}{\elleom 1}$$ and $$\norma{Xu-Xv}{\Mar{1^{*}}}\leq C\norma{f-g}{\elleom 1}.$$
    \end{itemize}
\end{cor}

\begin{cor}
    Let $\mu,\nu$ belong to $\mathcal{M}(\Omega)$ and let $u,v$ be the corresponding solutions to \eqref{Problemabase}. Then there exists a constant $C=C(Q, \Omega, \alpha)$ such that 
    $$\norma{u-v}{\Mar{1^{**}}}\leq C\norma{\mu-\nu}{\mathcal{M}(\Omega)}$$ 
    and 
    $$\norma{Xu-Xv}{\Mar{1^{*}}}\leq C\norma{\mu-\nu}{\mathcal{M}(\Omega)}.$$
\end{cor}

\begin{cor}
    Let $F,G\in\left(\elleom p\right)^m$ with $p\geq2$ and let $u,v$ be the corresponding solutions to \eqref{Problemabase}. Then there exists a constant $C=C(Q, \Omega,p, \alpha)$ such that 
    \begin{itemize}
        \item If $p>Q$, then $$\norma{u-v}{\elleom{\infty}}\leq C\norma{F-G}{\elleom p}$$ and $$\norma{u-v}{\SobomX2}\leq C\norma{F-G}{\elleom p};$$
        \item If $2\leq p <Q$, then $$\norma{u-v}{\elleom{p^{*}}}\leq C\norma{F-G}{\elleom p}$$ and $$\norma{u-v}{\SobomX2}\leq C\norma{F-G}{\elleom p}.$$
    \end{itemize}
\end{cor}

\begin{proof}
    For each of the previous corollaries, it suffices to note that $u-v$ satisfies (here we focus on Corollary \ref{Stability_Dati_fg})
    \begin{equation*}
        \begin{dcases}
            X^*\left(A(x) X (u-v)\right) = f-g \quad & \Omega\\
            u-v=0 \quad & \partial\Omega.
        \end{dcases}
    \end{equation*}
    Then, the thesis follows from the estimates given by Theorems \ref{Stampacchia-Linfty}, \ref{Stampacchia-p**-grande}, \ref{Stampacchia-p**-piccolo}, \ref{March_estimates}, \ref{BG_gradiente}, \ref{Stampacchia_p*divergenza}, \ref{Stampacchia_measuredata} and Lax Milgram's Theorem (depending on the type of data $f$ and $g$).
\end{proof}

\section{Comparison with the uniformly elliptic case and application to the Heisenberg Laplacian}\label{Sect_Heisenberg}

In general, the constant $Q$ can be expected to be greater than the dimension $N$. For instance, in the context of a Carnot group, there is an explicit expression for $Q$ (see \cite[Proposition 11.15]{HP-SobPoinc}) that shows $Q = N$ if and only if $m = N$ and the vector fields $X_j$ are linearly independent at each point $x$. Therefore, although our results extend classical elliptic estimates to a more general setting, they do not constitute an improvement over the classical ones. Indeed, when $Q > N > 2$, we have
$$\frac{2Q}{Q+2}>\frac{2N}{N+2},$$
which implies that we have a higher summability requirement to get finite-energy solutions. Moreover,
$$\frac{Qp}{Q-2p}<\frac{Np}{N-2p}\quad\forall p\in\left(1,\frac{N}{2}\right),$$
which means that the improvement in summability is worse than the one we have in the uniformly elliptic case. Also note that, if $\frac{N}{2}<p<\frac{Q}{2}$, the solution to a uniformly elliptic equation with datum in $\elleom p$ would be bounded, while the solution to an $X$-elliptic equation would only belong to $\frac{Qp}{Q-2p}$.

We now apply the previous results to the Heisenberg Laplacian, which is the differential operator on $\R^N=\R^{2n+1}$ given by
$$
\Delta_H:=\sum_{i=1}^n\left(X_i\right)^2+\left(Y_i\right)^2,
$$
where
$$
\left\{\begin{aligned}
X_i & =\frac{\partial}{\partial x_i}+2 y_i \frac{\partial}{\partial t}, \\
Y_i & =\frac{\partial}{\partial y_i}-2 x_i \frac{\partial}{\partial t}, \\
T & =\frac{\partial}{\partial t} .
\end{aligned}\right.
$$
In the following, we will refer to the family of vector fields $\{X_1,\ldots,X_n,Y_1,\ldots,Y_n\}$ as $\mathcal X$. 
Note that, since the vector fields $X_i, Y_i$ and $T$ satisfy the H\"ormander condition at step $1$, the operator $-\Delta_H$ (despite not being elliptic) is $X$-elliptic with respect to $\mathcal X$ and (by \cite[Proposition 11.15]{HP-SobPoinc}) $Q=2n+2$.

The results we have proved in \Cref{Sect_LpRegularity} imply the following theorem.

\begin{theo}
    Let $n\geq1$, $\Omega\subset\R^{2n+1}$ be an open set of finite Lebesgue measure, $f\in\elleom p$ with $p\geq 1$ (or $f\in\mathcal{M}(\Omega)$) and let $u$ be the solution (by duality, or found by approximation) of 
    \begin{equation}
        \begin{dcases}
            -\Delta_H u = f(x)& \Omega\\
            u=0 \quad & \partial\Omega.
        \end{dcases}
    \end{equation}
    Then $u$ satisfies the following properties:
    \begin{itemize}
        \item if $p=1$, or $f\in\mathcal{M}(\Omega)$, then $u\in\Mar{\frac{n+1}{n}}$ and $\mathcal X u\in\Mar{\frac{2n+2}{2n+1}}$;
        \item if $p\in(1,\frac{2n+2}{n+2})$, then $u\in\elleom{\frac{(n+1)p}{n+1-p}}$ and $\mathcal X u\in\elleom{\frac{(2n+2)p}{2n+2-p}}$;
        \item if $p\in[\frac{2n+2}{n+2}, n+1)$, then $u\in\elleom{\frac{(n+1)p}{n+1-p}}$ and $\mathcal X u\in\elleom{2}$;
        \item if $p\in(n+1, \infty]$, or $f\in\elleom{\frac{N}{2},1}$, then $u\in\elleom{\infty}$ and $\mathcal X u\in\elleom{2}$.
    \end{itemize}
\end{theo}
\begin{remark}
    Note that, for any $n\geq1$, we have
    $$1<\frac{2n+2}{n+2}<n+1.$$
    This is equivalent to the fact that the homogeneous dimension $Q=2n+2$ is strictly greater than $2$.
\end{remark}

\begin{remark}
    The previous result only utilizes the fact that the Heisenberg Laplacian is an $X$-elliptic operator, without taking full advantage of its specific properties. This implies that, similar to what can be achieved with the (classical) Laplacian, there may be potential for improving this result. Specifically, the summability of the term $\mathcal{X} u$ might be better than what we have indicated, at least for data belonging to $\elleom p$ with $p>\frac{2n+2}{n+2}$.
\end{remark}

\begin{remark}
    The estimates we have for measure data are sharp. For example, consider the homogeneous norm on the Heisenberg group (see \cite{Bramanti_book,LeDonne-book} for an introduction to Carnot groups)
     $$\norma{(x,y,t)}{H}=\sqrt[4]{\left(|x|^2+|y|^2\right)^2+t^2}$$
    and $\Omega=\{\norma{(x,y,t)}{H}\leq1\}$. Then, the solution to the problem 
    \begin{equation}
        \begin{dcases}
            -\Delta_H u = \delta_0& \Omega\\
            u=0 \quad & \partial\Omega
        \end{dcases}
    \end{equation}
    is given by
    $$u(x,t,y)=\Gamma_H(x,y,t)-1.$$
    Here    
    \begin{equation}
        \Gamma_H(x,y,t)= C_n\norma{(x,y,t)}{H}^{2-Q}=C_n\norma{(x,y,t)}{H}^{-2n}
    \end{equation}
    is the Green function for the Heisenberg Laplacian (whose explicit form was described by Folland in \cite{Folland_Green}).

    Note that $u\in\Mar{\frac{Q}{Q-2}}$. Indeed, we have
    $$\Gamma_H(x,y,t)\geq \lambda \iff \norma{(x,y,t)}{H}\leq C_n\lambda^{\frac{1}{2-Q}},$$
    which implies, $$
    \abs{\{\Gamma_H\geq \lambda\}}=\abs{\{\norma{(x,y,t)}{H}\leq C_n\lambda^{\frac{1}{2-Q}}\}}= \frac{\tilde C}{\lambda^{\frac{Q}{2-Q}}} $$
    for some $\tilde C=\tilde C(n)>0$.
    On the other hand, $u$ does not belong to $\elleom{\frac{Q}{Q-2}}$: indeed, we have
    \begin{equation}
        \begin{split}
            \io \Gamma_H(x,y,t)^{\frac{Q}{Q-2}}
    &=
    \int_0^\infty \abs{\{\Gamma_H^{\frac{Q}{Q-2}}\geq \lambda\}\cap\Omega}\rm d \lambda \\&
    = \int_1^\infty \abs{\bigg\{ (x,y,t)\in\Omega\,\mid\,\norma{(x,y,t)}{H}\leq C_n\lambda^\frac{-1}{Q}\bigg\}}\rm d \lambda
    \\&=
    \tilde C \int_1^\infty \lambda^{-1}\rm d \lambda=+\infty.
        \end{split}
    \end{equation}
    Note that similar computations can be performed whenever a description of the Green function on a Carnot group are available.
\end{remark}

\section*{Acknowledgements}
I wish to thank Luigi Orsina for his valuable feedback and his help in revising this manuscript.
I am also grateful to Marco Bramanti and Ermanno Lanconelli for kindly answering my questions regarding the existing literature on $X$-elliptic operators.\\
I declare no conflicts of interest. This research did not receive any specific grant from funding agencies in the public, commercial, or not-for-profit sectors.\\
I am a member of the GNAMPA group of INdAM.

\end{document}